\documentclass[12pt]{amsart}
\usepackage{amssymb,amsthm,amstext,amsmath,amscd,latexsym,amsfonts}
\usepackage[all]{xy}
\def\hR{\hat{R}}
\def\hRm{\hat{R_{\m}}}
\def\hRn{\hat{R_{\n}}}
\def\hRM{\hat{R}^{(J_M)}}
\def\hRTT{\hat{R}_{\TT}}
\def\hRJJ{\hat{R}_{\JJ}}
\def\ModR{\mathrm{Mod}(R)}

\def\modR{\mathrm{mod}(R)}
\def\ArtR{\mathrm{Art}(R)}

\def\WideR{\mathrm{Wid}_{R}}
\def\WideRR{\mathrm{Wid}_{\hat{R}}}

\def\Mod{\mathrm{Mod}}
\def\mod{\mathrm{mod}}
\def\Art{\mathrm{Art}}

\def\Wide{\mathrm{Wid}}

\def\subext{\mathrm{subext}}
\def\quotext{\mathrm{quotext}}
\def\Hom{\mathrm{Hom}}

\def\Ker{\mathrm{Ker}}

\def\Spec{\mathrm{Spec}}
\def\SpecR{\mathrm{Spec}\ R}

\def\Ass{\mathrm{Ass}}

\def\AssR{\mathrm{Ass}_{R}}
\def\Att{\mathrm{Att}}
\def\AttR{\mathrm{Att}_{R}}
\def\ann{\mathrm{ann}}

\def\Soc{\mathrm{Soc}}

\def\N{\mathbb N}

\def\m{\mathfrak m}
\def\n{\mathfrak n}
\def\l{\mathfrak l}
\def\p{\mathfrak p}
\def\P{\mathfrak P}
\def\q{\mathfrak q}
\def\Q{\mathfrak Q}

\def\TT{\mathcal T}
\def\JJ{\mathcal J}

\def\J{\mathcal J}

\def\X{\mathcal{X}}

\theoremstyle{plain} 
\newtheorem{theorem}{\textbf Theorem}[section]
\newtheorem{lemma}[theorem]{\textbf Lemma}
\newtheorem{corollary}[theorem]{\textbf Corollary}
\newtheorem{proposition}[theorem]{\textbf Proposition}

\theoremstyle{definition}
\newtheorem{definition}[theorem]{\textbf Definition}
\newtheorem{remark}[theorem]{\textbf Remark}

\numberwithin{equation}{section}

\begin{document}
\title[Remarks on subcategories of artinian modules]{\textbf{Remarks on subcategories of artinian modules}}
\author{Naoya Hiramatsu} 
\address{Department of general education, Kure National College of Technology, 2-2-11 Agaminami, Kure, Hiroshima, 737-8506 JAPAN}
\email{hiramatsu@kure-nct.ac.jp}
\subjclass[2000]{13C05, 16D90, 13J10}
\date{\today}
\keywords{artinian module, wide subcategory, Serre subcategory}
\begin{abstract}
We study two subcategories of the category of artinian modules, a wide subcategory and a Serre subcategory. 
We prove that all wide subcategories of artinian modules are Serre subcategories. 
We also provide the bijection between the set of Serre subcategories and the set of specialization closed subsets of the set of closed prime ideals of some completed ring. 
These results are artinian analogues of the theorems proved in \cite{T08}.    
\end{abstract}
\maketitle
\section{Introduction}

Classification theory of subcategories has been studied by many authors in many areas \cite{G62, H85, N92, Th97, Ho01, T08, K08}. 
In 1990's, Hopkins \cite{H85} and Neeman \cite{N92} classify thick subcategories of the derived categories of perfect complexes in terms of the ring spectra. 
Thomason \cite{Th97} generalizes this result to quasi-compact and quasi-separated schemes.
Now the classification theorem by them is known as the Hopkins-Neeman-Thomason theorem.  

Let us recall the definitions of several subcategories of an abelian category. 
We say that a full subcategory is wide if it is closed under kernels, cokernels and extensions. 
A Serre subcategory is defined to be a wide subcategory which is closed under subobjects.  
Let $R$ be a commutative noetherian ring and $M$ be an $R$-module. 
We denote by $\ModR$ the category of $R$-modules and $R$-homomorphisms and by $\modR$ the full subcategory consisting of finitely generated $R$-modules. 
We also denote by $\SpecR$ the set of prime ideals of $R$ and by $\Ass _R M$ the set of associated prime ideals of $M$.

Classifying subcategories of a module category also has been studied by many authors. 
Classically, Gabriel \cite{G62} gives a bijection between the set of Serre subcategories of $\modR$ and the set of specialization closed subsets of $\SpecR$. 
Recently, the following result was proved by Takahashi \cite{T08} and Krause \cite{K08}.
Takahashi \cite{T08} first proved the theorem and Krause \cite{K08} generalized it to the category of arbitrary modules which is closed under submodules, extensions and direct unions.

\begin{theorem}\cite[Theorem 4.1]{T08}\cite[Corollary 2.6]{K08}\label{Theorem B}
Let $R$ be a noetherian ring. 
Then we have the following 1-1 correspondences;
{\small 
$$
\begin{array}{ccc}
\begin{CD}
&\left\{ 
\begin{array}{c}
\text{subcategories of}\ \modR \ \text{closed under} \\ 
\text{submodules and extensions}\ 
\end{array}
\right\}
&\begin{matrix}
@>{\Psi }>> \\ @<<{\Phi }<
\end{matrix}
&\left\{ 
\begin{array}{c}
\text{subsets of}\ \SpecR 
\end{array}
\right\} \\
&\begin{matrix}
@AA{\subseteq}A 
\end{matrix}
&
&\begin{matrix}
@AA{\subseteq}A 
\end{matrix}
\\
&\left\{ 
\begin{array}{c}
\text{Serre subcategories}\\ 
\text{of}\ \modR 
\end{array}
\right\}
&\begin{matrix}
@>{\Psi }>> \\ @<<{\Phi }<
\end{matrix}
&\left\{ 
\begin{array}{c}
\text{specialization closed}\\
\text{subsets of}\ \SpecR 
\end{array}
\right\}
\end{CD}
\end{array}
$$
}
where $\Psi ({\mathcal M}) = \cup _{M \in {\mathcal M}} \Ass_{R} M$ and $\Phi ({\mathcal S}) = \{ M \in \mod | \Ass _{R} M \subseteq {\mathcal S} \}$. 
\end{theorem}

In addition, Takahashi \cite{T08} pointed out a property concerning wide subcategories of $\modR$. 
Actually he proved the following theorem. 

\begin{theorem}\cite[Theorem 3.1, Corollary 3.2]{T08}\label{Theorem A}
Let $R$ be a noetherian ring. 
Then every wide subcategory of $\modR$ is a Serre subcategory of $\modR$.
\end{theorem}

\noindent
It is worth nothing that Hovey \cite{Ho01} proved the theorem by using the Hopkins-Neeman-Thomason theorem, but in the case when $R$ is a quotient ring of a coherent regular ring by a finitely generated ideal.

In the present paper we want to consider the artinian analogue of these results. 

In section \ref{Coh}, we consider wide subcategories of artinian modules. 
We shall show that the artinian analogue of Theorem \ref{Theorem A} also holds.

\begin{theorem}[Theorem \ref{Theorem Coh}]
Let $R$ be a noetherian ring. 
Then every wide subcategory of $\ArtR$ is a Serre subcategory of $\ArtR$. 
\end{theorem}

In section \ref{Class}, we propose to classify Serre subcategories of artinian modules. 
We consider some completion of a ring (see Proposition \ref{isom}), so that all of artinian modules can be regarded as modules over it. 
We classify Serre subcategories in terms of a specialization closed subset of the set consisting of closed prime ideals of the completed ring.

\begin{theorem}[Theorem \ref{classification}]
Let $R$ be a noetherian ring. 
Then one has an isomorphism
$$
\begin{array}{l}
\{ \ \text{subcategories of}\ \ArtR \ \text{closed under quotient modules and extensions}\ \} \\ 
\qquad \cong \{ \ \text{subsets of the set consisting of closed prime ideals of }\ \hR \ \}.  
\end{array}
$$
Moreover this induces the isomorphism
$$
\begin{array}{l}
\{ \ \text{Serre subcategories of}\ \ArtR \ \} \\ 
\qquad \cong \left\{ 
\begin{array}{c}
\text{specialization closed subsets of}\\
\text{the set consisting of closed prime ideals of }\ \hR 
\end{array}
\right\}.
\end{array}
$$
\end{theorem}

In this paper, we always assume that $R$ is a commutative ring with identity, and by a subcategory we mean a nonempty full subcategory which is closed under isomorphism.  	

\section{wide subcategories of artinian modules}\label{Coh}

In this section, we investigate wide subcategories of artinian modules. 
First we recall the definitions of the categories.

\begin{definition}
A subcategory of an abelian category is said to be a wide subcategory if it is closed under kernels, cokernels and extensions. 
We also say that a subcategory is a Serre subcategory if it is a wide subcategory which is closed under subobjects.
\end{definition}

Let $M$ be an artinian $R$-module. 
We denote by $\Soc (M)$ the sum of simple submodules of $M$. 
Since $\Soc (M)$ is also artinian, there exist only finitely many maximal ideals $\m$ of $R$ for which $\Soc (M)$ has a submodule isomorphic to $R/\m$. 
Let the distinct such maximal ideals be $\m_1, \cdots ,  \m_s$. 
Set $J_M = \bigcap _{i = 1}^{s}\m $ and $\hRM = \varprojlim R/J_{M}^n$.

%
%
\begin{lemma}\cite[Lemma 2.2]{S92}\label{Sharp}
Each non-zero element $m \in M$ is annihilated by some power of $J_M$. 
Hence $M$ has the natural structure of a module over $\hRM$ in such a way that a subset of $M$ is an $R$-submodule if and only if it is an $\hRM$-submodule. 
\end{lemma}

\begin{proof}
Although a proof of the lemma is given in \cite{S92}, we need in the present paper how the $\hRM$-module structure is defined for an artinian module $M$. 
For this reason we briefly recall the proof of the lemma.

Since $\Soc (M) = \oplus _{i=1}^{s}(R/\m _{i})^{n_{i}}$, $M$ can be embedded in $\oplus _{i=1}^{s}(E_{R}(R/\m _{i}))^{n_{i}}$ where $E_{R}(R/\m )$ is an injective hull of $R/\m$. 
Note that an element of $E_{R}(R/\m)$ is annihilated by some power of $\m$.  
Hence one can show that each element of $M$ is annihilated by some power of $ \m _1 \cdots \m _s = J_M$.

Let $x \in M$ and $\hat{r} = (r_n + J_M ^n)_{n \in \N} \in \hRM$. 
Suppose that $J _M ^k x = 0$. 
It is straightforward to check that $M$ has the structure of an $\hRM$-module such that $\hat{r}x = r_{k}x$.

\end{proof}

\begin{remark}
As shown in the proof of Lemma \ref{Sharp},  $M$ can be embedded in $\oplus _{i=1}^{s}(E_{R}(R/\m _{i}))^{n_{i}}$. 
Thus the maximal ideals $\m_1, \cdots ,  \m_s$ are just associative prime ideals of $M$ since $\Ass _R \oplus _{i=1}^{s}(E_{R}(R/\m _{i}))^{n_{i}} = \Ass _R \Soc (M) = \{ \m_1, \cdots ,  \m_s \}$. 

\end{remark}

%
%
By virtue of Lemma \ref{Sharp}, each artinian $R$-module can be regarded as a module over some complete semi-local ring. 
We note that Matlis duality theorem is allowed over a noetherian complete semi-local ring (cf. \cite[Theorem 1.6]{O76}). 
It is the strategy of the paper that we replace the categorical property on a subcategory of finitely generated (namely, noetherian) modules with that of artinian modules by using Matlis duality.    
We denote by $\ArtR$ the subcategory consisting of artinian $R$-modules.

\begin{lemma}\label{duality}
Let $(R, \m _{1}, \cdots , \m _{s})$ be a noetherian complete semi-local ring and set $E = \oplus _{i = 1}^{s} E_{R}(R/\m_{i})$. 
For each subcategory $\X$ of $\ModR$, we denote by $\X ^\vee = \{ M^{\vee}\ |\ M \in \X \}$ where $(-)^\vee = \Hom _{R}(- , E)$. 
Then the following assertions hold.
\begin{itemize}
\item[(1)] If $\X$ is a subcategory of $\ArtR$ (resp. $\modR$) which is closed under quotient modules (resp. submodules) and extensions, then $\X ^\vee$ is a subcategory of $\modR$ (resp. $\ArtR$) which is closed under submodules (resp. quotient modules) and extensions.  

\item[(2)] If $\X$ is a wide subcategory of $\ArtR$ (resp. $\modR$), then $\X ^\vee$ is also a wide subcategory of $\modR$ (resp. $\ArtR$). 

\item[(3)] If $\X$ is a Serre subcategory of $\ArtR$ (resp. $\modR$), then $\X ^\vee$ is also a Serre subcategory of $\modR$ (resp. $\ArtR$). 
\end{itemize}
\end{lemma}

\begin{proof}
Since Matlis duality theorem is allowed over a noetherian complete semi-local ring, the assertions hold by Matlis duality.

\end{proof}

\begin{definition}\label{coh step}
Let $M$ be an $R$-module. 
For a nonnegative integer $n$, we inductively define a subcategory $\Wide ^{n}_{R} (M)$ of $\ModR$ as follows:

\begin{itemize}
\item[(1)] Set $\WideR ^{0}(M) = \{ M \}$.

\item[(2)] For $n \geq  1$, let $\WideR ^{n} (M)$ be a subcategory of $\ModR$ consisting of all $R$-modules $X$ having an exact sequence of either of the following three forms: 
$$
\begin{array}{l}
A \to B \to X \to 0, \\
0 \to X \to A \to B, \\ 
0 \to A \to X \to B \to 0 
\end{array}
$$
where $A, B \in \WideR ^{n-1} (M)$.
\end{itemize}
\end{definition}

\begin{remark}\label{remark coh step}
Let $M$ be an $R$-module and $n$ be a nonnegative integer. 
Then the following hold. 
\begin{itemize}
\item[(1)] There is an ascending chain $\{ M \} = \WideR ^{0}(M) \subseteq \WideR ^{1} (M) \subseteq \cdots \subseteq \WideR ^{n} (M) \subseteq \cdots \subseteq \WideR (M)$ of subcategories $\ModR$, where $\WideR (M)$ is the smallest wide subcategory of $\ModR$ which contains $M$.

\item[(2)] $\bigcup _{n \geq 0}\WideR ^{n}(M)$ is wide and the equality $\WideR (M) = \bigcup _{n \geq 0}\WideR ^{n}(M)$ holds.
\end{itemize}
\end{remark}

\begin{definition}
Let $J$ be an ideal of $R$. 
For each $R$-module $M$, we denote by $\Gamma _{J} (M)$ the set of elements of $M$ which are annihilated by some power of $J$, namely $\Gamma _{J} (M) = \bigcup _{n \in \N}( 0 :_M J^n )$.
An $R$-module $M$ is said to be $J$-torsion if $M = \Gamma _{J} (M)$. 
We denote by $\Mod _{J}(R)$ the subcategory consisting of $J$-torsion $R$-modules. 
\end{definition}

\begin{lemma}\label{lemma J-torsion}
For each object $M$ in $\Mod _{J}(R)$, $M$ has the structure of an $\hat{R}^{(J)}$-module where $\hat{R}^{(J)}$ is a $J$-adic completion of $R$. 
\end{lemma}

\begin{remark}
By using an inductive argument on $n$, we can show that if $M$ is artinian (resp. $J$-torsion), then $\bigcup _{n \geq 0}\WideR ^{n}(M)$, hence $\WideR (M)$, is a subcategory of $ \ArtR$ (resp. $\Mod _{J}(R)$) since $\ArtR$ (resp. $\Mod _{J}(R)$) is a wide subcategory. 
\end{remark}

\begin{corollary}\label{corollary coh}
Let $M$ be an artinian $R$-module. 
Then $\WideR (M)$ and $\Wide _{\hRM} (M)$ are equivalent as subcategories of $\Art (\hRM )$. 
\end{corollary}

\begin{proof}
As remarked above, since $M$ is $J_M$-torsion, we can naturally identify $\WideR (M)$ with a subcategory of $\Mod _{J_M}(R)$. 
It is also a subcategory of $\Art (\hRM )$ by Lemma \ref{lemma J-torsion}. 

\end{proof}

\begin{theorem}\label{Theorem Coh}
Let $R$ be a noetherian ring. 
Then every wide subcategory of $\Art (R)$ is a Serre subcategory of $\Art (R)$. 
\end{theorem}

\begin{proof}
Let $\X$ be a wide subcategory of $\ArtR$. 
It is sufficiently to show that $\X$ is closed under submodules. 
Assume that $\X$ is not closed under submodules. 
Then there exists an $R$-module $X$ in $\X$ and $R$-submodule $M$ of $X$ such that $M$ does not belong to $\X$. 
Applying Lemma \ref{Sharp} to $X$, $X$ is a module over the complete semi-local ring $\hR : = \hR ^{(J_X)}$ and $M$ is an $\hat{R}$-submodule of $X$. 
Now we consider the wide subcategory $\WideR (X)$. 
By virtue of Corollary \ref{corollary coh}, $\WideR (X) = \WideRR (X)$ as a subcategory of $\Art (\hR)$. 
Since $\hat{R}$ is a complete semi-local ring, by Matlis duality, we have the equivalence of the categories $\WideRR (X) \cong \{ \WideRR (X)^{\vee } \} ^{op} \cong \WideRR (X^{\vee })^{op}$ where $( - ) ^{\vee} = \Hom _{\hat{R}}(-, E_{\hat{R}}(\hat{R}/J_{X}\hat{R}))$. 
Since $\WideRR (X^{\vee })$ is a wide subcategory of finitely generated $\hat{R}$-modules, it follows from Theorem \ref{Theorem A} that $\WideRR (X^{\vee })$ is a Serre subcategory. 
Thus $M^{\vee}$ is contained in $\WideRR (X^{\vee })$. 
Using Matlis duality again, we conclude that $M$ must be contained in $\WideRR (X) = \WideR (X)$, hence also in $\X$. 
This is a contradiction, so that $\X$ is closed under submodules. 

\end{proof}

\section{Classifying subcategories of artinian modules}\label{Class}

In this section, we shall give the artinian analogue of the classification theorem of subcategories of finitely generated modules (Theorem \ref{classification}).  
First, we state the notion and the basic properties of attached prime ideals which are play a key role of our theorem. 
For the detail, we recommend the reader to refer to \cite{S76, S92} and \cite[\S 6 Appendix.]{M}.

\begin{definition}\label{def of secondary}
Let $M$ be an $R$-module. 
We say that $M$ is secondary if for each $a \in R$ the endomorphism of $M$ defined by the multiplication map by $a$ is either surjective or nilpotent. 
\end{definition}

\begin{remark}
If $M$ is secondary then $\p = \sqrt{\ann _R (M)}$ is a prime ideal and $M$ is said to be $\p$-secondary.
\end{remark}

\begin{definition}
$M = S_1 + \cdots + S_r$ is said to be a secondary representation if $S_i$ is a secondary submodule of $M$ for all $i$. 
And we also say that the representation is minimal if the prime ideals $\p _i  = \sqrt{\ann _R (S_i)}$ are all distinct, and none of the $S_i$ is redundant 
\end{definition}

\begin{definition}\label{def of attached prime}
A prime ideal $\p$ is said to be an attached prime ideal of $M$ if $M$ has a $\p$-secondary quotient. 
We denote  by $\AttR M$ the set of the attached prime ideals of $M$. 
\end{definition}

\begin{remark}\label{remark attach}
Let $M$ be an $R$-module. 
\begin{itemize}
\item[(1)] If $M = S_1 + \cdots + S_r$ is a minimal representation and $p _i = \sqrt{\ann _R (S_i )}$. then $\AttR M = \{ \p _1 , \cdots , \p _r \}$. 
See \cite[Theorem 6.9]{M}.

\item[(2)]Let $M$ be an $R$-module. 
Given a submodule $N \subseteq M$, we have
$$
\AttR M/N \subseteq \AttR M \subseteq \AttR (N) \cup \AttR M/N. 
$$
See \cite[Theorem 6.10]{M}.

\item[(3)]It is known that if $M$ is artinian then $M$ has a secondary representation. 
Thus it has a minimal one. 
See \cite[Theorem 6.11]{M}. 

\end{itemize}
\end{remark}

Let $(R, \m)$ be a noetherian local ring and $\X$ a Serre subcategory of $\ArtR$. 
By virtue of Lemma \ref{Sharp}, $\ArtR$ is equivalent to $\Art (\hR )$ where $\hR$ is an $\m$-adic completion of $R$. 
Now we consider $\X$ as a subcategory of $\Art (\hR)$. 
Since $\X^\vee$ is a Serre subcategory of $\mod (\hR)$ (Lemma \ref{duality}), $\X ^{\vee}$, hence $\X$, corresponds to the specialization closed subset of $\Spec \ \hR$ by Theorem \ref{Theorem B}. 
Namely there is the bijection between the set of Serre subcategories of $\ArtR$ and the set of specialization closed subsets of $\Spec \ \hR$.  
This observation provides us that we should consider a larger set than $\Spec \ R$ to classify subcategories of artinian modules.

\vspace{12pt}

In the rest of this section, we always assume that $R$ is a noetherian ring.
\vspace{12pt}

As mentioned in Lemma \ref{Sharp}, we can determine some complete semi-local rings for each artinian module respectively, so that the artinian module has the module structure over such a completed ring. 
Now we attempt to treat all of artinian $R$-modules as modules over the same completed ring.  
For this, we consider the following set of ideals of $R$: 
$$
\TT = \{ \ I \ | \  \text{a length of}\ R/I \ \text{is finite}\ \} . 
$$

The set $\TT$ forms a directed set ordered by inclusion. 
Then we can consider the inverse system $\{ R/I, f_{I, I'} \}$ where $f_{I, I'}$ are natural surjections. 
That is, 
$$
I, I' \in \TT \ \text{and}\  I' \subseteq I\Rightarrow  f_{I,I'}: R/I' \to R/I .
$$
We denote $\varprojlim _{I \in \TT}R/I $ by $\hat{R}_{\TT}$.

\begin{lemma}\label{extended sharp}
Every artinian $R$-module has the structure of an $\hRTT$-module in such a way that a subset of an artinian $R$-module $M$ is an $R$-submodule if and only if it is an $\hRTT$-submodule. 
Consequently, we have an equivalence of categories $\Art (R) \cong \Art (\hRTT )$. 
\end{lemma}

\begin{proof}
The proof of the first part of the lemma will go through similarly to the proof of Lemma \ref{Sharp}. 
The last part of the lemma holds from the definition of the $\hRTT$-module structure. 
\end{proof}

We set another family of ideals of $R$ as  
$$
\JJ = \{ \ \m _{1}^{k_1} \cdots \m_{s}^{k_s} \ | \ \m _i  \ \text{is a maximal ideal of}\ R, \ k_i \in \N \} . 
$$ 
It is also a directed set ordered by inclusion and we denote by $\hRJJ$ its inverse limit on the system via natural surjections.

\begin{proposition}\label{1st isom}
There is an isomorphism of topological rings; 
$$
\hRTT \cong \hRJJ .
$$  
\end{proposition}

\begin{proof}
Let $I$ be an ideal in $\TT$. 
Note that $\Ass _R R/I = \{ \m _1 , \cdots , \m _s \}$ for some maximal ideals $\m _i$ of $R$. 
Since $\m _i$ are finitely generated, there exists a positive integer $k$ such that $(\m _1 \cdots \m_s )^k \subseteq I$. 
Thus, for each ideal $I$ in $\TT$, we can take some ideal $J$ in $\JJ$ such that $J \subseteq I$. 
Hence $\TT$ and $\JJ$ give the same topology on $R$, so that $\hRTT \cong \hRJJ$ as topological rings.  

\end{proof}

Now we consider a direct product of rings 
$$
\prod _{\n \in \max (R)} \hRn 
$$
where $\max (R)$ is the set of maximal ideals of $R$ and $\hRm$ is an $\m$-adic completion of $R$. 
We regard the ring as a topological ring by a product topology, namely the linear topology defined by ideals which are of the form $\m _{1}^{k_1}\hR _{\m _{1}} \times \cdots \times \m _{s}^{k_s}\hR _{\m _{s}} \times \prod _{\n \not= \m_1, \cdots , \m_s} \hR _{\n} $ for some $\m _{i} \in \max (R)$ and $k_i \in \N$.

\begin{proposition}\label{2nd isom}\cite[\S 2.13. Proposition 17]{Bou}
There is an isomorphism of topological rings 
$$
\hRJJ \cong \prod _{\n \in \max (R)} \hRn .
$$
\end{proposition}

\begin{proof}
Let $J$ be an ideal in $\JJ$ and suppose that $J = \m _{1}^{k_1} \cdots \m _{s}^{k_s}$. 
Note that $R/J$ is isomorphic to $\prod _{i=1}^{s}R/\m _{i}^{k_i}$ by Chinese remainder theorem. 
Let us set $A = \prod _{\n \in \max (R)} \hRn$. 
For all $J \in \J$, we define mappings $\varphi _{J} : A \to R/J$ by 
$$
\varphi _{J} : A \to R/J \cong \prod _{i=1}^{s}R/\m _{i}^k ; \quad (\hat{a}_\m )  \to (\bar{a_{\m _1 ^{k_1}}},  \cdots , \bar{a_{\m _s ^{k_s}}}). 
$$
Here we denote $(a _{\m ^k} + \m ^k) \in \hRm$ by $\hat{a}_\m $. 
It is easy to see that $\varphi = \{ \varphi _{J} \} _{J \in \JJ}$ is a morphism from $A$ to $\hRJJ$. 
Write $p _J : \hRJJ \to R/J$ for the projection.
We note that the topology of $\hRJJ$ coincides with the linear topology defined by $\{ \Ker \ p _J \} _{J \in \JJ}$ (cf. \cite[\S 8]{M}). 
Set $V_J = \ker \ p_J$.
For each $V_J$, we take the open set $W_J = \m _{1}^{k_1}\hR _{\m _{1}} \times \cdots \times \m _{s}^{k_s}\hR _{\m _{s}} \times \prod _{\n \not= \m_1, \cdots , \m_s} \hR _{\n} $ in $A$.   
Then $p_J \circ \varphi (W_J ) = 0$. 
Thus $\varphi (W_J ) \subseteq V_J$, so that $\varphi$ is continuous.

For each ideal $J \in \JJ$, we take an ideal $W_J$ of $A$ as above. 
As mentioned before, $A$ has a linear topology defined by $\{ W_J \} _{J \in \JJ}$, and $\varprojlim A/W_J = A$. 
We define mappings $\psi _{J} : \hRJJ \to A/W_J$ by 
$$
\psi _{J} : \hRJJ \to A/W_J \cong \prod _{i=1}^{s}R/\m _{i}^k \cong R/J ; \quad (a_J + J)_{J \in \JJ} \to \bar{a_{J}}. 
$$
We also see that $\psi _{J}$ induces the morphism $\psi = \{ \psi _{J} \} _{J \in \JJ}: \hRJJ \to A$ which is a continuous mapping. 
In fact, $W_J$ is just a kernel of the natural projection $A \to A/W_J$ and $\psi (V_J )$ goes to $0$ via the projections.  

Finally, we shall show $\varphi \circ \psi = 1_{\hRJJ}$ and $\psi \circ \varphi = 1_{A}$, but this is clear from the definition of $\varphi$ and $\psi$. 

\end{proof}

Combining Proposition \ref{1st isom} with Proposition \ref{2nd isom}, we can show the following.

\begin{corollary}\label{isom}
There are isomorphisms of topological rings 
$$
\hRTT \cong \hRJJ \cong \prod _{\n \in \max (R)} \hRn .
$$  
\end{corollary}

For closed prime ideals of $\prod _{\n \in \max (R)} \hRn$, we have the following result.

\begin{proposition}\label{closed prime}
Every proper closed prime ideal of $\prod _{\n \in \max (R)} \hRn$ is of the form $\p \times \prod _{\n \in \max (R), \m \not= \n  } \hR _{\n}$ for some prime ideal $\p \in \Spec \ \hRm$. 
Hence we can identify the set of closed prime ideals of $\prod _{\n \in \max (R)} \hRn$ with the disjoint union of $\Spec \ \hRm$, i.e. $\coprod _{\n \in \max (R)} \Spec \ \hRn$.
\end{proposition}

\begin{proof}
Let us set $A = \prod _{\n \in \max (R)} \hRn$. 
We take an element $\hat{e}_\m = (\hat{e}_{\m, \n})$ of $A$ defined by 
$$
\hat{e}_{\m, \n}  =
\left\{
\begin{array}{ll}
\hat{1},&\text{if} \ \m = \n , \\
0,&\text{otherwise}.
\end{array}
\right.
$$
Let $\P$ be an arbitrary closed prime ideal of $A$.
Since $\hat{e}_\l \cdot \hat{e}_\m = 0$, we have $\hat{e}_\l \cdot \hat{e}_\m \in \P$. 
Thus if there is a maximal ideal $\m$ such that $\hat{e}_\m$ is not contained in $\P$, $\hat{e}_\l$ is in $\P$ whenever $\l$ is not equal to $\m$.   

Suppose that $\hat{e}_\m \not \in \P$. 
Then $\P$ contains $\hat{e}_\l$ for all maximal ideals $\l \not = \m$. 
First we shall show the family $(\hat{e}_{\l} )$ where $\l$ runs through all maximal ideals of $R$ except $\m$ is summable in $A$. 
Namely the sum $\varepsilon = \Sigma \hat{e}_\l$ is an element of $A$. 
For each neighborhood $W_J $ (of $0$) in $A$, we take a finite set of maximal ideals $H_{J} = \{ \m_1 , \cdots , \m_s , \m \}$. 
Then one can show that 
$$
\Sigma _{\l \in H} \hat{e}_\l \in W _J
$$
for every finite set of maximal ideals $H$ which does not intersect with $H_{J}$. 
Thus our claim follows from Cauchy's criterion (\cite[Chapter 3, \S 5, no. 2, Theorem 1]{Bou98}).   
Note that $\varepsilon$ is contained in $\P$ since $\P$ is a closed ideal. 
Thus $\varepsilon A = 0 \times \prod _{\n \not = \m } \hR _{\n}$ is an $A$-submodule of $\P$. 
Then we have the sequence:
$$
0 \to \P / \varepsilon A \to A/ \varepsilon A \cong \hR _{\m} \to A/\P \to 0.
$$   
Since $\P / \varepsilon A$ is a prime ideal of $\hRm$, we conclude that $\P$ is of the form $\p \times \prod _{\n \in \max (R), \m \not= \n  } \hR _{\n}$ for some prime ideal $\p \in \Spec \ \hRm$. 

Suppose that all of elements $\hat{e}_\m$ are contained in $\P$. 
Then we can easily show that $\P = A$ and this is a contradiction. 

\end{proof}


By virtue of Corollary \ref{isom}, we can equate the rings $\hRTT$, $\hRJJ$ and $\prod _{\n \in \max (R)} \hRn$. 
In the rest of this paper we always denote them by $\hR$ and identify the set of closed prime ideals of $\hR$ with $\coprod _{\n \in \max (R)} \Spec \ \hRn$.

Let $M$ be an artinian $R$-module. 
It follows from Lemma \ref{extended sharp} that $M$ is also an artinian $\hR$-module.

\begin{proposition}\label{closedness} 
Let $M$ be an artinian $R$-module. 
Then $\ann _{\hR} (M)$ is a closed ideal of $\hR$. 
\end{proposition}

\begin{proof}
We denote by $U_I$ a kernel of the natural projection $\hRTT \to R/I$ for each ideal $I \in \TT$. 
It suffices to prove that the inclusion $\ann _{\hR} (M) \supseteq \cap _{I \in \TT} (\ann _{\hR} (M) + U_I)$ holds. 
Take an arbitrary element $\hat{a} \in \cap _{I \in \TT} (\ann _{\hR} (M) + U_I )$. 
Then there exist some elements $\hat{b}_{I} \in \ann _{\hR} (M)$ and $\hat{c}_{I} \in U_I$ such that $\hat{a} = \hat{b}_{I} + \hat{c}_{I}$ for all $I$. 
Let $x$ be an element of $M$. 
Then there exists some ideal $I \in \TT$ such that $Ix = 0$. 
Thus
$$
\hat{a}x = (\hat{b}_{I} + \hat{c}_{I})x = \hat{c}_{I}x = c_{I, I} x \subseteq Ix = 0. 
$$
Hence $\hat{a}$ is an annihilator of $M$. 

\end{proof}

\begin{remark}\label{closedness radiacl}
Under the same assumption in Proposition \ref{closedness}, a radical of $\ann _{\hR}(M)$ is also a closed ideal. 
In fact, let $\hat{a}$ be an element of a closure of $\sqrt{\ann _{\hR}(M)}$. 
Take $x \in M$ and suppose that $Ix = 0$ for some $I \in \TT$. 
Since $\hat{b}_{I} \in \sqrt{\ann _{\hR} (M)}$, $\hat{b}_{I}^k \in \ann _{\hR} (M)$. 
Hence we see that $\hat{a}^k x = (\hat{b}_{I} + \hat{c}_{I})^k x = 0$ holds, so that $\hat{a} \in \sqrt{\ann _{\hR} (M)}$.
Consequently, $\Att _{\hR} M$ is a subset of the set of closed prime ideals of $\hR$. 
\end{remark}

\begin{lemma}\cite[Exercise 8.49]{S00}\label{torsion decomp}
Let $M$ be an artinian $R$-module. 
Set $\AssR M = \{ \m _{1}, \dots , \m _{s} \}$.
Then $M$ is direct sums of $\Gamma _{\m _i}(M)$, namely $M = \oplus _{i=1}^{s}\Gamma _{\m _i}(M)$. 
\end{lemma}

\begin{proof}
It is clear that $M$ contains $\Sigma _{i=1}^{s} \Gamma _{\m _i}(M)$.

For each element $x \in M$, there is some positive integer $k$ such that $J^k x = 0$ where $J = \m _{1} \dots \m _{s}$. 
Since $\m _{1}, \dots \m _{s}$ are all distinct maximal ideals, we have 
$$
\m _{1}^{k} \m _{2}^{k} \cdots \m _{s-1}^{k} + \m _{1}^{k} \m _{2}^{k} \cdots \m _{s-2}^{k}\m _{s}^{k} + \cdots + \m _{2}^{k} \m _{3}^{k} \cdots \m _{s}^{k} = R.
$$
Thus there are elements $r_i \in \m _{1}^{k}\cdots \m _{i-1}^{k} \m _{i+1}^{k} \cdots \m _{s}^{k}$ such that $\Sigma _{i=1}^{s}r_i = 1$, and we get the equality $x = \Sigma _{i=1}^{s}r_{i}x$. 
Then we can show that each $r_i x$ is an element of $\Gamma _{\m _i}(M)$. 
In fact,  
$$
\m _{i}^{k}r_{i}x \subseteq \m _{1}^{k} \m _{2}^{k} \cdots \m _{s}^{k}x = 0. 
$$
Therefore we obtain $M = \Sigma _{i=1}^{s} \Gamma _{\m _i}(M)$.

It remains to show the sum above is a direct sum. 
This follows from the facts that $\AssR (\Gamma _{\m _{i}}(M) ) = \{ \m _i \}$ and all $\m _i$ are distinct.

\end{proof}

Let $M$ be an $\m$-torsion $R$-module. 
Then $M$ has the structure of an $\hR$-module and an $\hRm$-module. 
Note that the $\hRm$-module action on $M$ is identical with the action by means of the natural inclusion $\hRm \to \prod _{\n \in \max (R)} \hRn \cong \hR$. 
We also note from Lemma \ref{extended sharp} or Lemma \ref{Sharp} that $N$ is an $\hR$-submodule (resp. a quotient $\hR$-module) of $M$ if and only if it is an $\hRm$-submodule (resp. a quotient $\hRm$-module) of $M$.

\begin{proposition}\label{prop A}
Let $M$ be an $\m$-torsion $R$-module. 
Then 
$$\Att _{\hR} M = \Att _{\hRm} M$$ 
as a subset of $\coprod _{\n \in \max (R)} \Spec \ \hat{R_{\n }}$. 
\end{proposition}

\begin{proof}
Let $\P \in \Att _{\hR} M$ and $W$ be a $\P$-secondary quotient $\hR$-module of $M$. 
Note that $W$ is also a quotient $\hRm$-module of $M$. 
As noted in Remark \ref{closedness radiacl}, $\P = \sqrt{\ann _{\hR} (W)}$ is a closed prime ideal. 
Thus $\P$ is of the form $\p \times \prod _{\l \not= \n} \hR _\n$ where $\p$ is a prime ideal of $\hR _\l$ for some maximal ideal $\l$. 
First we shall show $\l = \m$. 
For this, we show $\hat{e}_\n \in \P$ if $\n \not= \m$ (see Proposition \ref{closedness} for the definition $\hat{e}_\n$). 
Let $x$ be an element of $W$ and suppose that $\m ^k x = 0$. 
The $\hR$-module action $\hat{r}x$ is defined by $r_{\m^k}x$ for each $\hat{r} \in \hR$. 
Then we see that  
$$
\hat{e}_\n x = \hat{e}_{\n , \m}x = 0_{\m^k} x = 0.
$$ 
Thus, if $\n \not= \m$, $e _\n$ is contained in $\ann _{\hR} (W)$, so that in $\sqrt{\ann _{\hR} (W)} = \P$. 
Hence $\hat{e}_\n \hR = \hR _{\n} \subseteq \P $ whenever $\n \not= \m$, so that $\l$ must be $\m$. 
Consequently, $\p$ is a prime ideal of $\hRm$.

Since the $\hRm$-action on $W$ is the same as the action via the natural inclusion $\hRm \to \hR$, we have $\sqrt{\ann _{\hR} (W)} \cap \hRm = \sqrt{\ann _{\hRm}(W)}$.  
Therefore $\p \in \Att _{\hRm} M$.

Conversely, let $\q$ be an attached prime ideal of $M$ as $\hRm$ modules and $V$ be a $\q$-secondary quotient $\hRm$-module of $M$. 
Then $V$ is also an $\hR$-quotient module of $M$, and $\Q = \q \times \prod _{\n \not= \m}\hR _{\n}$ is equal to $\sqrt{\ann _{\hR} (V)}$. 
Hence $\Q \in \Att _{\hR} M$. 

\end{proof}

Combing Proposition \ref{prop A} with Lemma \ref{torsion decomp}, we have the following corollary. 

\begin{corollary}\label{cor A}
Let $M$ be an artinian $R$-module. 
Then 
$$
\Att _{\hR} M = \coprod _{\m \in \AssR M}\Att _{\hRm} \Gamma _{\m} (M ) 
$$ 
as a subset of $\coprod _{\n \in \max (R)} \Spec \ \hat{R_\n}$.
\end{corollary}

%
%
Let us state the result which is a key to classify the subcategory of the category of noetherian modules.

\begin{theorem}\cite[Corollary 4.4]{T08}\cite[Corollary 2.6]{K08}\label{Krause}
Let $M$ and $N$ be finitely generated $R$-modules. 
Then $M$ can be generated from $N$ via taking submodules and extension if and only if $\Ass _{R} M \subseteq \Ass _{R} N$. 
\end{theorem}

The following lemma is due to Sharp \cite{S76}.

\begin{lemma}\cite[3.5.]{S76}\label{att ass}
Let $(R, \m _{1}, \cdots , \m _{s})$ be a commutative noetherian complete semi-local ring and set $E = \oplus _{i = 1}^{s} E_{R}(R/\m_{i})$.
For an artinian $R$-module $M$, we have  
$$
\Att _{R} M = \Ass _{R} \Hom _R (M, E). 
$$ 
\end{lemma}

The next claim is reasonable as the artinian analogue of Theorem \ref{Krause}.

%
%
\begin{theorem}\label{theorem 1}
Let $M$ and $N$ be artinian $R$-modules. 
Then $M$ can be generated from $N$ via taking quotient modules and extensions as $R$-modules if and only if $\Att _{\hR} M \subseteq \Att _{\hR} N$. 
\end{theorem}

\begin{proof}
Suppose that $M$ is contained in $\quotext _R (N)$. 
It is clear from the property of attached prime ideals (Remark \ref{remark attach}) that $\Att _{\hR} M \subseteq \Att _{\hR} N$ holds.

Conversely, suppose that $\Att _{\hR} M \subseteq \Att _{\hR} N$.
First, we shall show that we may assume that $M$ and $N$ are $\m$-torsion $R$-modules for some maximal ideal $\m$. 
In fact, $M$ (resp. $N$) can be decomposed as $M = \oplus _{\m \in \AssR M}\Gamma _{\m }(M)$ (resp. $N = \oplus _{\n \in \AssR N}\Gamma _{\n }(N)$) and the assumption implies that $\Att _{\hRm} \Gamma _{\m} (M ) \subseteq \Att _{\hRm} \Gamma _{\m} (N )$ for all $\m \in \AssR M$ by Corollary \ref{cor A}.   
If we show that $\Gamma _{\m} (M )$ is contained in $\quotext _R (\Gamma _{\m} (N ))$, we can get the assertion since $\quotext _R (N)$ is closed under direct sums and direct summands.

Let $M$ and $N$ be $\m$-torsion $R$-modules and $E$ be an injective hull of $\hRm /\m \hRm $ as an $\hRm$-module. 
Since $M$ and $N$ are also artinian $\hRm$-modules, $M^{\vee}$ and $N^{\vee}$ are finitely generated $\hRm$-modules by Matlis duality, where $(-)^\vee = \Hom _{\hRm}(- , E)$. 
Since $\Att _{\hRm} M$ (resp. $\Att _{\hRm} N$) is equal to $\Ass _{\hRm} M^{\vee}$ (resp. $\Ass _{\hRm} N^{\vee}$) (Lemma \ref{att ass}), the inclusion 
$$
\Ass _{\hRm} M^{\vee} \subseteq \Ass _{\hRm} N^{\vee}
$$
holds.   
By virtue of Theorem \ref{Krause}, we conclude that $M^{\vee}$ can be generated from $N^{\vee}$ via taking submodules and extensions, i.e. $M^{\vee} \in \subext _{\hRm}(N^{\vee})$. 
Hence it follows from Matlis duality and Lemma \ref{duality} that  
$$
M^{\vee \vee} \cong M \in \subext _{\hRm}(N^{\vee})^{\vee} = \quotext _{\hRm}(N). 
$$
Since artinian $\hRm$-modules are also artinian $R$-modules (cf. Lemma \ref{Sharp}), we conclude that $M \in \quotext _{R}(N)$.

\end{proof}

We define by $\Psi $ the map sending a subcategory $\X$ of $\ArtR$ to 
$$
\Att \X = \cup _{M \in \X}\Att _{\hR} M
$$ 
and by $\Phi$ the map sending a subset $S$ of $\coprod _{\n \in \max (R)} \Spec \ \hat{R_\n}$ to 
$$
\{ M \in \ArtR \ | \Att _{\hR} M \subseteq S \}.
$$

Note from Corollary \ref{cor A} that $\Psi (\X )$ determine the subset of $\coprod _{\n \in \max (R)} \Spec \ \hat{R_\n}$. %
On the other hand, it follows from Remark \ref{remark attach} (2) that $\Phi (S)$ is closed under quotient modules and extensions.

\vspace{8pt}
Now we state a main theorem of this section.

%
%
\begin{theorem}\label{classification}
Let $R$ be a commutative noetherian ring. 
Then $\Psi$ and $\Phi$ induce a bijection between the set of subcategories of $\ArtR$ which are closed under quotient modules and extensions and the set of subsets of $\coprod _{\n \in \max (R)} \Spec \ \hat{R_\n}$. 

Moreover, they also induce a bijection between the set of Serre subcategories of $\ArtR$ and the set of specialization closed subsets of $\coprod _{\n \in \max (R)} \Spec \ \hat{R_\n}$. 
\end{theorem}
%
%
\begin{proof}
Let $\X$ be a subcategory of $\ArtR$ which is closed under quotient modules and extensions. 
The subcategory $\Phi \Psi (\X )$ consists of all artinian $R$-modules $M$ with $\Att _{\hR} M \subseteq \cup _{X \in \X}\Att _{\hR} X$. 
It is clear that $\X$ is a subcategory of $\Phi \Psi (\X )$. 
Let $M$ be an artinian $R$-module with $\Att _{\hR} M \subseteq \cup _{X \in \X}\Att _{\hR} X$. 
For each ideal $\P \in \Att _{\hR} M$, there exists $X^{(\P)} \in \X$ such that $\P \in \Att _{\hR} X^{(\P)} $. 
Take the direct sums of such objects, say $X = \oplus _{\P \in \Att _{\hR} M} X^{(\P)}$, $X$ is also an object of $\X$. 
In fact, $\Att _{\hR} M$ is a finite set and $\X$ is closed under finite direct sums.  
It follows from the definition of $X$ that $\Att _{\hR} M \subseteq \Att _{\hR} X$. 
By virtue of Theorem \ref{theorem 1}, $M$ is contained in $\quotext _R (X)$, so that  in $\X$. 
Hence we have the equality $\X = \Phi \Psi (\X )$.

Let $S$ be a subset of $\coprod _{\n \in \max (R)} \Spec \ \hat{R_\n}$. 
It is trivial that the set $\Psi \Phi (S)$ is contained in $S$. 
Let $\p$ be a prime ideal in $S$. 
Take a maximal ideal $\m$ so that $\p$ is a prime ideal of $\hRm$. 
We consider an $\hRm$-module $E_{\hRm /\p \hRm }(\hRm / \m \hRm)$. 
Then we have the equality: 
$$
\Att _{\hRm } E_{\hRm /\p \hRm}(\hRm / \m \hRm) = \Ass _{\hRm } \hRm /\p \hRm = \{ \p \}.
$$
Note that $E_{\hRm}(\hRm / \m \hRm)$ is artinian as an $R$-module. 
Indeed, we have the equality $E_{\hRm}(\hRm / \m \hRm) = E_{R}(R / \m R)$ as $R$-modules (\cite[Theorem 18.6 (iii)]{M}). 
Since $E_{\hRm /\p \hRm }(\hRm / \m \hRm)$ is an $\hRm$-submodule (thus an $R$-submodule) of $E_{\hRm}(\hRm / \m \hRm)$, it is an artinian $R$-module. 
Hence $E_{\hRm /\p \hRm}(\hRm / \m \hRm)$ is an artinian $R$-module which is a $\p$-secondary $\hRm$-module. 
Consequently, $E_{\hRm /\p \hRm }(\hRm / \m \hRm)$ belongs to $\Phi (S)$, so that $\p \in \Psi \Phi (S)$.

\vspace{8pt}
Suppose that $\X$ is a Serre subcategory of $\ArtR$. 
Let $\p$ be a prime ideal of $\hRm$ which is contained in $\Psi (\X )$.    
Chose $\q \in \Spec \ \hRm$ such that $\p \subseteq \q$. 
Then we have the inclusion of $\hRm$-modules (hence, of $R$-modules):
$$
0 \to E_{\hRm /\q \hRm}(\hRm/ \m \hRm) \to E_{\hRm /\p \hRm}(\hRm/ \m \hRm). 
$$
Since $E_{\hRm /\p \hRm}(\hRm/ \m \hRm)$ is an artinian $R$-module which is a $\p$-secondary $\hRm$-module, $E_{\hRm /\p \hRm}(\hRm/ \m \hRm)$ is contained in $\X$. 
Thus $E_{\hRm /\q \hRm}(\hRm/ \m \hRm)$ is also in $\X$ since $\X$ is closed under submodules. 
Hence we have that $\q \in \Psi (\X )$, so that $\Psi (\X )$ is closed under specialization. 

On the other hand, let $S$ be a specialization closed subset of $\coprod _{\n \in \max (R)} \Spec \ \hat{R_\n}$. 
We shall show $\Phi (S)$ is a Serre subcategory. 
Since $\Phi (S)$ is closed under quotient modules and extensions, we sufficiently show that it is closed under submodules. 
Let $M$ be in $\Phi (S)$ and $N$ be an $R$-submodule of $M$. 
Set $J_M = \cap _{\m \in \Ass _{R} M}\m$ and $\hRM = \varprojlim R/J_{M}^n$. 
Then $M$ is an artinian $\hRM$-module and $N$ is also an artinian $\hRM$-submodule of $M$ (Lemma \ref{Sharp}). 
Since $\hRM$ is a complete semi-local ring, Matlis duality is allowed. 
By using Matlis duality, we can show that $N^{\vee}$ is contained in the Serre subcategory generated by $M^{\vee}$, where $(-)^\vee = \Hom _{\hRM}(- , E_{\hRM}(\hRM/J_{M}\hRM))$. 
Thus, by Theorem \ref{Theorem B}, $\Ass _{\hRM} N^{\vee}$ is in $\cup _{\p \in \Ass _{\hRM} M^{\vee}} V(\p)$. 
Since we have the equalities $\Att _{\hR} N = \Att _{\hRM} N = \Ass _{\hRM} N^{\vee}$, and $\cup _{\p \in \Ass _{\hRM} M^{\vee}} V(\p) = \cup _{\p \in \Att _{\hRM} M} V(\p) \subseteq S$, $\Att _{\hR} N$ is contained in $S$. 
Therefore $N$ is in $\Phi (S)$. 

\end{proof}

\subsection*{Acknowledgments}
The author express his deepest gratitude to Ryo Takahashi and Yuji Yoshino for valuable discussions and helpful comments.


\end{document}